\documentclass[11pt,reqno]{amsart}
\usepackage[T1]{fontenc}
\usepackage{lmodern}
\usepackage{amsmath,amssymb,mathtools}
\usepackage{microtype}
\usepackage{needspace}
\usepackage{tikz,booktabs}
\usepackage[a4paper,margin=30mm]{geometry}
\usepackage[hidelinks]{hyperref}
\hypersetup{
  pdftitle={Paths maximize the expected range of graph-indexed random walks},
  pdfauthor={Yinfeng Zhu},
  pdfsubject={The BHM expectation theorem and its LNR corollary},
  pdfkeywords={graph homomorphisms, integer Lipschitz functions, expected range}
}
\newtheorem{theorem}{Theorem}[section]
\newtheorem{proposition}[theorem]{Proposition}
\newtheorem{lemma}[theorem]{Lemma}
\newtheorem{corollary}[theorem]{Corollary}
\theoremstyle{remark}
\newtheorem{remark}[theorem]{Remark}
\numberwithin{equation}{section}
\newcommand{\Z}{\mathbb Z}
\newcommand{\E}{\mathbb E}
\newcommand{\PP}{\mathbb P}
\newcommand{\HH}{\mathcal H}
\newcommand{\LL}{\mathcal L}
\newcommand{\one}{\mathbf 1}
\DeclareMathOperator{\Bin}{Bin}
\DeclareMathOperator{\rk}{r}
\DeclareMathOperator{\diver}{div}
\title[Expected ranges of graph-indexed random walks]{Paths maximize the expected range\texorpdfstring{\\}{ }of graph-indexed random walks}
\author{Yinfeng Zhu}
\thanks{Independent Researcher.}
\date{}
\begin{document}
\begin{abstract}
We prove that a path maximizes the expected range of a uniformly chosen
graph homomorphism into the integers, with one vertex pinned at zero,
among all connected bipartite graphs of the same order. This establishes
the expectation form of the Benjamini--H\"aggstr\"om--Mossel conjecture.
The proof restricts and rescales a homomorphism on each bipartition class,
then contracts the edges on which the resulting height function is
constant. A quantitative estimate for the rank of these zero edges
compensates for a parity term in the expected range of a simple random
walk, allowing an induction on the number of vertices. We then prove
that the BHM inequality implies the Loebl--Ne\v set\v ril--Reed inequality
for uniformly chosen integer 1-Lipschitz functions on arbitrary connected
graphs, and hence obtain the LNR conjecture as a corollary of BHM.
The proof was obtained through interaction with OpenAI GPT-6 Astra
and verified by the author. The main results have also been formalized
and checked in Lean~4.
\end{abstract}
\maketitle

\section{Introduction}\label{sec:intro}

Let $G=(V,E)$ be a finite connected simple graph and fix a vertex $o\in V$.
A \emph{standard height function} is a map $f:V\to\Z$ satisfying
\[
 f(o)=0,\qquad |f(u)-f(v)|=1\quad(uv\in E).
\]
Denote this set by $\HH(G,o)$. It is nonempty exactly when $G$ is bipartite.
Indeed, the parity of a standard height function is a proper two-colouring;
conversely, distance from the root is a standard height function on a
connected bipartite graph.
A \emph{lazy height function} satisfies the same root condition and
$|f(u)-f(v)|\le1$ on every edge; write $\LL(G,o)$ for its set.
Both sets are finite whenever they are nonempty, since
$|f(v)|\le\operatorname{dist}_G(o,v)$.
For either model, define the range by
\[
 R(f)=\max_{v\in V}f(v)-\min_{v\in V}f(v).
\]
The image of a height function on a connected graph is an integer interval.
Consequently, the alternative convention $|f(V)|$ increases every range
and every expected range by one and leaves all comparisons below unchanged.

Taking expectations with respect to the uniform measures, define
\[
 \widehat h(G)=\E_{\HH(G,o)}R(f),\qquad
 h(G)=\E_{\LL(G,o)}R(f).
\]
These expectations are independent of $o$: the map $f\mapsto f-f(o')$
gives a range-preserving bijection between the height functions rooted
at $o$ and those rooted at $o'$.
We suppress the root from the notation when its choice is immaterial.
The path $P_n$ has $n$ vertices. On a path, the standard model coincides
with a simple random walk with independent increments in $\{-1,1\}$;
the lazy model has independent increments uniform on $\{-1,0,1\}$.

Benjamini, H\"aggstr\"om, and Mossel~\cite{BHM} introduced the standard
graph-indexed model and conjectured that paths maximize its expected range.
The following theorem establishes this assertion for all connected bipartite graphs.

\begin{theorem}[BHM expectation conjecture]\label{thm:bhm}
For every finite connected bipartite simple graph $G$ on $n$ vertices,
\begin{equation}\label{eq:bhm}
 \widehat h(G)\le\widehat h(P_n).
\end{equation}
\end{theorem}

Loebl, Ne\v set\v ril, and Reed~\cite{LNR} considered the lazy model,
which is defined on every connected graph. Their corresponding path
extremality conjecture follows from Theorem~\ref{thm:bhm}.

\begin{corollary}[LNR expectation conjecture]\label{cor:lnr}
For every finite connected simple graph $G$ on $n$ vertices,
\begin{equation}\label{eq:lnr}
 h(G)\le h(P_n).
\end{equation}
If $G$ contains a cycle, the inequality is strict.
\end{corollary}

The BHM conjecture also has a stronger formulation: the entire range
distribution should be stochastically dominated by that of a path.
Theorem~\ref{thm:bhm} establishes the expectation formulation.
Wu, Xu, and Zhu~\cite{WXZ} proved both expectation inequalities for trees.
Bok and Ne\v set\v ril~\cite{BN} extended them to unicyclic graphs.
For the distributional question, Berger, Ji, and Metz~\cite{BJM} proved
stochastic domination for all trees in the lazy model and for spiders
in the standard model. Their introduction also records the expectation
and distributional forms of the BHM conjecture. Our proof does not rely
on the extremal results for these special graph classes.

\subsection*{The proof mechanism}
For a lazy height function $g$ on a connected graph $S$, let
\[
 A_g=\{uv\in E(S):g(u)=g(v)\}
\]
be its full set of zero edges. Contract the connected components of
$(V(S),A_g)$ and delete loops and repeated edges. The resulting graph is
bipartite, because $g$ descends to a standard height function on it.
If the sampling weight of $g$ depends only on $A_g$, then, conditional
on the full set $A_g$, the descended function is uniform among the
standard height functions on the quotient.

To prove the BHM theorem, we first restrict $f$ to each class of a bipartition
$V=B\sqcup W$ and divide the height differences by two. This produces
weighted lazy functions $g_B,g_W$ on two smaller auxiliary graphs.
Their ranges satisfy the pointwise identity
\[
 R(f)=1+R(g_B)+R(g_W).
\]
This identity reduces the induction to a comparison between the numbers
of zero-edge components in the auxiliary graphs and in the path model.

The key comparison retains the surplus in the mean rank of the zero
edges. A lower bound on the probability that all edges of a specified
forest are zero yields a generating-function inequality that quantifies
this surplus. We compare the expected range of a simple random walk
with a positive-kernel upper envelope whose error is supported on
positive even times. A cycle in the auxiliary graph supplies enough
rank surplus to absorb the error. Odd cycles are
handled by a pointwise rank inequality. For bipartite auxiliary graphs,
the marginal weight factors over edges, and a finite Fourier argument
compares constant and nonconstant boundary values on an even cycle.

Section~\ref{sec:decimation} develops the reduction to the two auxiliary
graphs. Section~\ref{sec:comparison} proves the analytic comparison.
Section~\ref{sec:bhmproof} establishes the rank surplus and completes
the proof of BHM. Only then, in Section~\ref{sec:lnr}, do we prove that
BHM implies LNR. The latter implication uses a separate zero-edge rank
estimate for the uniform lazy model.

The graph $K_{2,3}$ serves as a running example for the auxiliary
graphs, their induced measures, and the contraction argument.

\subsection*{Formal verification and code availability}
Theorem~\ref{thm:bhm} and Corollary~\ref{cor:lnr}, including strictness
for graphs containing a cycle, have been formalized in Lean~4~\cite{Lean4}
with mathlib~\cite{Mathlib}. The formal proofs were checked using
Lean~4.34.0, without \texttt{sorry} placeholders or axioms beyond Lean's
standard classical foundations. The source code, pinned dependencies,
and reproduction instructions are available in the accompanying
repository~\cite{GIRWLean}:
\begin{center}
\small\url{https://github.com/yinfengzhu7-oss/graph-indexed-random-walk-lean}
\end{center}
The verified version is commit
\texttt{1af0f5a281b7} on the \texttt{main} branch. The formalization
uses equivalent arguments for some intermediate steps; the repository
documents their correspondence with the proof below.

\section{Decimation and zero-edge contraction}\label{sec:decimation}

For a graph $S$ and $A\subseteq E(S)$, let $k_S(A)$ be the number of
components of the spanning subgraph $(V(S),A)$, including isolated
vertices, and define its graphic rank by
\[
 \rk_S(A)=|V(S)|-k_S(A).
\]
We omit the subscript if the underlying graph is fixed.

\Needspace{8\baselineskip}
\begin{lemma}[Conditioning on the full zero set]\label{lem:fullzero}
Let $S$ be connected and let $g\in\LL(S)$ have probability proportional
to a positive weight $\omega(A_g)$ depending only on its full zero set.
For every zero set $A$ of positive probability, the simple quotient $S/A$
is connected and bipartite. Conditional on $A_g=A$, the descended height
function is uniform on $\HH(S/A)$ and has range $R(g)$. In particular,
\begin{equation}\label{eq:quotient-average}
 \E R(g)=\E\widehat h(S/A_g).
\end{equation}
\end{lemma}

\begin{proof}
The function $g$ is constant on each component of $(V(S),A)$.
Since the zero set is full, every original edge within such a component
belongs to $A$. Across distinct components, the height difference has
absolute value one. The descended function is therefore standard, and
its parity defines a bipartition of $S/A$.

Conversely, any rooted standard height function on $S/A$ pulls back to
a lazy height function whose full zero set is exactly $A$. Descent and
pullback are mutually inverse and preserve the range. All functions in
the fibre have the same weight $\omega(A)$, proving conditional uniformity. The
one-vertex quotient, with its unique height function, is included.
\end{proof}

Now let $G$ be connected and bipartite, with at least two vertices and
bipartition $V=B\sqcup W$. The \emph{auxiliary graph} $S_B$ has vertex
set $B$; two distinct vertices are adjacent if they have a common neighbour
in $W$. Define $S_W$ similarly. Both auxiliary graphs are connected,
where a one-vertex graph is regarded as connected. Choose $b_0\in B$ and
$w_0\in W$. For a standard height function on the original graph, define
\[
 g_B(u)=\frac{f(u)-f(b_0)}2\quad(u\in B),\qquad
 g_W(v)=\frac{f(v)-f(w_0)}2\quad(v\in W).
\]
Vertices in the same bipartition class have heights of the same parity.
It follows that $g_B$ and $g_W$ are integer-valued lazy height functions
on their auxiliary graphs.

\begin{lemma}[The single-class marginal]\label{lem:marginal}
For a uniform standard height function on $G$, the law of $g_B$ is
\begin{equation}\label{eq:marginal}
 \PP(g_B=g)\ \propto\ 2^{t_B(g)},\qquad
 t_B(g)=|\{w\in W:g\text{ is constant on }N_G(w)\}|.
\end{equation}
Every rooted lazy height function on $S_B$ occurs. The weight in
\eqref{eq:marginal} depends only on its full zero set.
\end{lemma}

\begin{proof}
Reroot at $b_0$. Each neighbourhood $N_G(w)$ is a clique in $S_B$, so its
$g$-values either all equal some $j$, or include both $j$ and $j+1$.
In the first case, $f(w)$ can be $2j-1$ or $2j+1$; in the second it must
be $2j+1$. Once the black heights are fixed, the choices for distinct
vertices $w$ can be made independently, because $W$ is an independent
set. Counting these extensions proves the stated marginal law and
shows that every $g$ occurs. Whether a neighbourhood is constant is
determined by its internal zero edges; a singleton neighbourhood is
automatically constant. Thus the weight depends only on $A_g$.
\end{proof}

Write $Q_B=S_B/A_{g_B}$ and $Q_W=S_W/A_{g_W}$, and let $K_B,K_W$
denote their numbers of vertices.

\begin{proposition}[Range decomposition]\label{prop:decimation}
For every standard height function on $G$,
\begin{equation}\label{eq:range-decomposition}
 R(f)=1+R(g_B)+R(g_W).
\end{equation}
Consequently,
\begin{equation}\label{eq:decimation}
 \widehat h(G)=1+\E\widehat h(Q_B)+\E\widehat h(Q_W).
\end{equation}
Each quotient has fewer vertices than $G$.
\end{proposition}

\begin{proof}
Let $M_B,m_B$ and $M_W,m_W$ be the extrema of $f$ on the two classes.
Every vertex has a neighbour in the opposite class, so
$|M_B-M_W|\le1$ and $|m_B-m_W|\le1$. Since the two classes have opposite
height parities, both absolute differences equal one. It follows that
\[
 R(f)=\frac{M_B+M_W+1}{2}-\frac{m_B+m_W-1}{2}
      =1+R(g_B)+R(g_W).
\]
Taking expectations and applying Lemmas~\ref{lem:fullzero}
and~\ref{lem:marginal} gives the expected-range identity.
Finally, $K_B\le |B|<|V|$ and $K_W\le |W|<|V|$.
\end{proof}

\paragraph{Running example: a graph with five vertices.}
Let $G_\star=K_{2,3}$, with bipartition
$B=\{u_0,u_1\}$ and $W=\{v_0,v_1,v_2\}$, rooted at $u_0$.
Every vertex in one class is adjacent to every vertex in the other.
Thus $S_B$ is a single edge and $S_W$ is a triangle: an auxiliary
graph need not be bipartite, even though $G_\star$ is.
For the restrictions, take $b_0=u_0$ and $w_0=v_0$.
The standard height function
\[
 \bigl(f_\star(u_0),f_\star(u_1);
       f_\star(v_0),f_\star(v_1),f_\star(v_2)\bigr)
       =(0,0;-1,1,1)
\]
induces $g_B=(0,0)$ and $g_W=(0,1,1)$, with coordinates in the
listed vertex orders. Its full zero sets and quotients are
\[
 A_{g_B}=\{u_0u_1\},\qquad A_{g_W}=\{v_1v_2\},\qquad
 Q_B\cong P_1,\qquad Q_W\cong P_2.
\]
In $Q_W$, the two edges from $v_0$ to $v_1,v_2$ become a single edge
after repeated edges are removed. The descended heights are $0$ on
$Q_B$ and $0,1$ on $Q_W$, as shown in
Figure~\ref{fig:running-example}. The range decomposition reads
$R(f_\star)=2=1+0+1$.

\begin{figure}[htbp]
\centering
\begin{tikzpicture}[
  x=1cm,y=1cm,font=\small,
  vertex/.style={circle,draw,fill=white,inner sep=0pt,minimum size=5.5pt},
  blackvertex/.style={vertex,fill=black},
  zeroedge/.style={line width=1.6pt},
  every label/.style={font=\small,inner sep=2pt}]
  \path[use as bounding box] (-1.9,-2.05) rectangle (11.25,2.65);
  \node at (0,2.35) {(a) $G_\star$ and $f_\star$};
  \draw (-1.15,0.55)--(0,1.6)--(1.15,0.55)
        (-1.15,0.55)--(0,0.55)--(1.15,0.55)
        (-1.15,0.55)--(0,-0.5)--(1.15,0.55);
  \node[blackvertex,label=left:{$u_0:0$}] at (-1.15,0.55) {};
  \node[blackvertex,label=right:{$u_1:0$}] at (1.15,0.55) {};
  \node[vertex,label=above:{$v_0:-1$}] at (0,1.6) {};
  \node[vertex,label=below:{$v_1:1$}] at (0,0.55) {};
  \node[vertex,label=below:{$v_2:1$}] at (0,-0.5) {};
  \node at (0,-1.5) {$R(f_\star)=2$};

  \begin{scope}[xshift=4.6cm]
    \node at (0,2.35) {(b) $S_B$ and $Q_B$};
    \draw[zeroedge] (-0.85,1.1)--(0.85,1.1);
    \node[vertex,label=above:{$u_0:0$}] at (-0.85,1.1) {};
    \node[vertex,label=above:{$u_1:0$}] at (0.85,1.1) {};
    \draw[->,>=stealth] (0,0.55)--(0,-0.7);
    \node[vertex,label=below:{$\{u_0,u_1\}:0$}] at (0,-1.3) {};
  \end{scope}

  \begin{scope}[xshift=9.1cm]
    \node at (0,2.35) {(c) $S_W$ and $Q_W$};
    \draw (0,1.55)--(-0.9,0.25) (0,1.55)--(0.9,0.25);
    \draw[zeroedge] (-0.9,0.25)--(0.9,0.25);
    \node[vertex,label=above:{$v_0:0$}] at (0,1.55) {};
    \node[vertex,label=below:{$v_1:1$}] at (-0.9,0.25) {};
    \node[vertex,label=below:{$v_2:1$}] at (0.9,0.25) {};
    \draw[->,>=stealth] (0,-0.35)--(0,-0.9);
    \draw (-0.95,-1.3)--(0.95,-1.3);
    \node[vertex,label=below:{$\{v_0\}:0$}] at (-0.95,-1.3) {};
    \node[vertex,label=below:{$\{v_1,v_2\}:1$}] at (0.95,-1.3) {};
  \end{scope}
\end{tikzpicture}
\caption{The running example $G_\star=K_{2,3}$. Labels give each vertex
or contracted set and its height. Filled vertices in (a) form $B$;
open vertices form $W$. Heavy edges in (b) and (c) are the full zero
sets of the restrictions. Arrows indicate contraction and removal
of loops and repeated edges.}
\label{fig:running-example}
\end{figure}
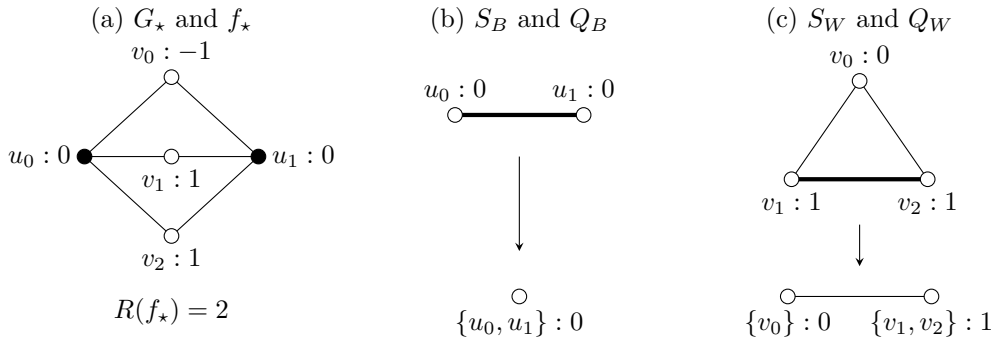

\paragraph{The induced measures in the running example.}
Now let $f$ be uniform on $\HH(G_\star,u_0)$.
If $f(u_1)=0$, the three white heights can be chosen independently
from $\{-1,1\}$, giving eight functions. If $f(u_1)=2$ or $-2$,
all three white heights are forced, giving one function in each case.
These are the ten possible functions. Accordingly, the three rooted
functions on $S_B$, namely $(0,0)$ and $(0,\pm1)$, have weights
$8,1,1$ in Lemma~\ref{lem:marginal}.

On $S_W$, the constant-zero function has weight $2^2=4$,
because both black neighbourhoods are constant. Each of the six
nonconstant rooted lazy functions has weight one. They are obtained
by assigning $0,1$ to the three vertices, using both values, and
then subtracting the value at $v_0$. Thus the induced measure on
this triangle is not the uniform lazy measure:
\[
 \PP(g_W\equiv0)=\frac25,\qquad
 \PP(g_W=g)=\frac1{10}
 \quad\text{for each nonconstant }g\in\LL(S_W,v_0).
\]
For example, conditional on $A_{g_W}=\{v_1v_2\}$, the two functions
$(0,1,1)$ and $(0,-1,-1)$ are equally likely. They descend to the
two rooted standard functions on $Q_W\cong P_2$, illustrating
Lemma~\ref{lem:fullzero}. Conditioning only on
$g_W(v_1)=g_W(v_2)$ also admits the constant function, which then
has conditional probability $2/3$. This distinguishes a prescribed
zero edge from a prescribed full zero set.
Since every nonconstant restriction has
range one, Proposition~\ref{prop:decimation} gives
\[
 \widehat h(G_\star)
   =1+\E R(g_B)+\E R(g_W)
   =1+\frac15+\frac35=\frac95.
\]

For the next estimate, we condition only on specified edges being zero;
we impose no condition on the remaining edges.

\begin{lemma}[Common-neighbour reflection]\label{lem:reflection}
If $u$ and $v$ are distinct vertices of a connected bipartite graph with
a common neighbour, then a uniform standard height function satisfies
\begin{equation}\label{eq:reflection}
 \PP(f(u)=f(v))\ge\frac12.
\end{equation}
For every forest $F\subseteq E(S_B)$,
\begin{equation}\label{eq:bhm-forest}
 \PP(F\subseteq A_{g_B})\ge 2^{-|F|}.
\end{equation}
\end{lemma}

\begin{proof}
Reroot at a common neighbour $w$ and condition on the entire function
$|f|$. On each component of the subgraph induced by $\{|f|>0\}$, the
sign of $f$ is constant: opposite nonzero signs would create an edge
difference of at least two. Conversely, any assignment of signs to
these components preserves the edge conditions: at a boundary edge,
the nonzero endpoint has absolute height one. Every sign assignment
therefore gives a legal function with the prescribed absolute values.
Uniformity of the original model makes the conditional signs
independent and fair.

Both $u$ and $v$ have absolute height one. If they lie in the same nonzero
component they have equal height; otherwise they have equal height with
probability $1/2$. This proves~\eqref{eq:reflection}.

For~\eqref{eq:bhm-forest}, prescribe the equalities along the edges of
$F$ one at a time. Conditioning on previous equalities is equivalent to
identifying the corresponding vertices in $B$ in the original graph $G$.
The resulting standard model remains uniform on a connected bipartite
graph. Since $F$ is a forest, the endpoints of its next edge are still
distinct; their original common neighbour remains. Apply
\eqref{eq:reflection} at each step and multiply the conditional bounds.
Repeated edges may be removed, but distinct vertices in $W$, even those
with identical neighbourhoods, are retained throughout.
\end{proof}

\section{From zero-edge rank to expected range}\label{sec:comparison}

Let $b_m$ be the expected range of an $m$-step simple symmetric random walk
started at zero, including time zero in the range. Thus
$b_m=\widehat h(P_{m+1})$ and $b_0=0$. The comparison in this section
will be applied with two binomial parameters, first $1/2$ and later $2/3$.

\subsection{A generating-function inequality retaining the mean surplus}

\begin{lemma}\label{lem:pgf}
Let $S$ be a connected graph with $d+1$ vertices, where $d\ge1$, and
let $A$ be a random subset of its edges. Suppose that, for a fixed
$p\in[0,1]$,
\begin{equation}\label{eq:forest-assumption}
 \PP(F\subseteq A)\ge(1-p)^{|F|}
 \quad\text{for every forest }F\subseteq E(S).
\end{equation}
Put
\[
 q=\rk(A),\qquad X=d-q,\qquad
 \Delta=\E q-(1-p)d.
\]
Then $\Delta\ge0$ and, for $0<s<1$,
\begin{equation}\label{eq:pgf}
 \E s^X-(1-p+ps)^d\ge\Delta(1-s)s^{d-1}.
\end{equation}
The endpoint inequalities follow by continuity, with $0^0=1$.
\end{lemma}

\begin{proof}
Fix a spanning tree $T$ of $S$ and set $Z=|A\cap E(T)|$ and $U=d-Z$.
Since $Z\le q$, we have $X\le U$ and
\[
 s^X-s^U=(1-s)\sum_{i=X}^{U-1}s^i
          \ge(q-Z)(1-s)s^{d-1}.
\]
An empty sum is zero. Expanding a product over the tree edges gives
\begin{align*}
 \E s^U-(1-p+ps)^d
 &=\sum_{F\subseteq E(T)}s^{d-|F|}(1-s)^{|F|}
     \bigl[\PP(F\subseteq A)-(1-p)^{|F|}\bigr]\\
 &\ge (1-s)s^{d-1}
       \sum_{e\in E(T)}[\PP(e\in A)-(1-p)].
\end{align*}
Every bracket is nonnegative by~\eqref{eq:forest-assumption}, so we may
retain only the singleton terms. Taking expectations in the first
display and adding the resulting inequalities proves~\eqref{eq:pgf}.
Also $\E q\ge\E Z\ge(1-p)d$.
\end{proof}

\subsection{A positive-kernel envelope for the path}

\begin{lemma}[Path range and its envelope]\label{lem:envelope}
The sequence $(b_m)$ is increasing and concave, with
\begin{equation}\label{eq:path-increments}
 b_{m+1}-b_m=2^{-m}\binom{m}{\lfloor m/2\rfloor}\quad(m\ge0).
\end{equation}
Set $B_0=0$ and, for $m\ge1$, define
\begin{equation}\label{eq:envelope}
 B_m=\alpha+\int_0^1(1-s^m)\nu(s)\,ds,\qquad
 \alpha=\frac4\pi-1,\quad
 \nu(s)=\frac{4s}{\pi(1-s^2)^{3/2}}.
\end{equation}
Then $a_m=B_m-b_m\ge0$, and more precisely
\begin{equation}\label{eq:parity-error}
 a_m=
 \begin{cases}
 0,&m=0\text{ or }m\text{ is odd},\\[2pt]
 \displaystyle\frac2\pi\int_0^1
       \frac{\sqrt{1-s}}{(1+s)^{3/2}}s^m\,ds,
       &m\ge2\text{ is even}.
 \end{cases}
\end{equation}
For $d\ge2$, write
\begin{equation}\label{eq:J}
 J_d=B_d-B_{d-1}
     =\frac4\pi\int_0^1
       \frac{s^d}{(1+s)^{3/2}\sqrt{1-s}}\,ds>0.
\end{equation}
\end{lemma}

\begin{proof}
Let $(S_j)$ be the walk and $M_m=\max_{0\le j\le m}S_j$.
Symmetry gives $b_m=2\E M_m$. The maximum increases at step $m+1$
exactly when $S_m=M_m$ and the next increment is $+1$, so
$b_{m+1}-b_m=\PP(S_m=M_m)$.
Reversing the first $m$ increments identifies this event with the event
that all partial sums are nonnegative. For a nonnegative endpoint $j$
of the correct parity, reflection at the first visit to $-1$ gives the
probability
$\PP(S_m=j)-\PP(S_m=j+2)$ for nonnegative paths ending at $j$.
The sum over $j$ telescopes, giving~\eqref{eq:path-increments}.
These increments are positive and nonincreasing: the increments at
indices $2r-1$ and $2r$ agree, and the next one is smaller.

Let $\rho(x)=1/(\pi\sqrt{1-x^2})$ for $-1<x<1$.
The substitution $x=\cos\theta$ shows that the odd moments vanish and
that the $2r$-th moment equals $4^{-r}\binom{2r}{r}$.
Combining these moments with~\eqref{eq:path-increments} and summing a
finite geometric series yields
\begin{equation}\label{eq:path-integral}
 b_m=\int_{-1}^1(1-x^m)\frac{1+x}{1-x}\rho(x)\,dx.
\end{equation}
For $0<s<1$, put $L(s)=(1+s)/(1-s)$ and $\ell(s)=(1-s)/(1+s)$.
Direct calculation gives
\[
 (L-\ell)\rho=\nu,
 \qquad 2\int_0^1\ell(s)\rho(s)\,ds=\frac4\pi-1=\alpha.
\]
For the second identity, $s=\cos\theta$ reduces the integral to
$\pi^{-1}\int_0^{\pi/2}\tan^2(\theta/2)\,d\theta$ before multiplication
by two. Split~\eqref{eq:path-integral} at zero and subtract it from
\eqref{eq:envelope}. For $m\ge1$ the difference is
\[
 \int_0^1\ell(s)\bigl[s^m+(-s)^m\bigr]\rho(s)\,ds,
\]
which proves~\eqref{eq:parity-error}; $m=0$ is separate.
Finally, subtract consecutive envelope values to obtain~\eqref{eq:J}.
All integrals are finite. In particular, near $s=1$ the factor $1-s^m$
is $O(1-s)$, which makes its product with $\nu$ integrable.
\end{proof}

For $Y_{d,p}\sim\Bin(d,p)$, set
\[
 \varepsilon_{d,p}=\E a_{Y_{d,p}}.
\]
The following estimate converts a rank surplus into the required range
comparison.

\begin{proposition}\label{prop:comparison}
Under the assumptions of Lemma~\ref{lem:pgf}, if $d\ge2$, then
\begin{equation}\label{eq:range-comparison}
 \E b_X-\E b_{Y_{d,p}}
 \le\varepsilon_{d,p}-\Delta J_d-\E a_X.
\end{equation}
\end{proposition}

\begin{proof}
Use the unified expression
$B_m=\alpha(1-0^m)+\int_0^1(1-s^m)\nu(s)\,ds$, with $0^0=1$.
The variables take finitely many values and each envelope integral is
finite. Linearity of the integral therefore permits us to interchange
integration and expectation. Integrating~\eqref{eq:pgf} against the
positive kernel $\nu$ gives
\[
 \E B_{Y_{d,p}}-\E B_X\ge\Delta J_d.
\]
The atom at zero contributes a nonnegative term, since the endpoint
inequality in~\eqref{eq:pgf} gives $\PP(X=0)\ge\PP(Y_{d,p}=0)$.
Substituting $B_m=b_m+a_m$ and rearranging completes the comparison.
\end{proof}

\paragraph{Running example: zero-edge rank and the surplus.}
Return to the induced measure on the triangle $S_W$.
Every edge is zero with probability $3/5$, and any specified two-edge
forest is zero with probability $2/5$, in agreement with the lower
bounds $1/2$ and $1/4$ from Lemma~\ref{lem:reflection}.
The constant restriction has three zero edges but rank two;
each nonconstant restriction has one zero edge and rank one.
For $d=2$ and $p=1/2$, this gives
\[
 \PP(q=2)=\frac25,\qquad \PP(q=1)=\frac35,\qquad
 \E q=\frac75,\qquad \Delta=\frac25.
\]
Thus $X=K_W-1=2-q$ is zero with probability $2/5$ and one otherwise.
The generating-function comparison can be read off exactly:
\[
 \E s^X-\left(\frac{1+s}{2}\right)^2
   =\frac25(1-s)s+\frac3{20}(1-s)^2
   \ge\Delta(1-s)s\qquad(0\le s\le1).
\]
This example also makes the parity correction visible. Since $X$
only takes the values zero and one, $a_X=0$, whereas $Y_{2,1/2}$
can equal two. In particular,
\[
 \E b_X=\frac35<\frac78=\E b_{Y_{2,1/2}},\qquad
 \varepsilon_{2,1/2}=\frac{a_2}{4}=\frac2\pi-\frac58.
\]
Using $J_2=8/\pi-2$, the upper bound in
Proposition~\ref{prop:comparison} is
$\varepsilon_{2,1/2}-\Delta J_2=7/40-6/(5\pi)<0$.
The rank surplus therefore more than compensates for the positive
parity error in this comparison.

\subsection{Bounding the parity error}

\begin{lemma}\label{lem:budget}
For $d\ge2$ and $1/2\le p\le2/3$,
\begin{equation}\label{eq:general-budget}
 \frac{\varepsilon_{d,p}}{J_d}
 <\frac{1}{8p^{3/2}}
   \frac{\Gamma(d-1/2)\Gamma(d+3/2)}{\Gamma(d+1)^2}.
\end{equation}
In particular,
\begin{equation}\label{eq:half-budget}
 \varepsilon_{d,1/2}<\frac{J_d}{2d}\qquad(d\ge2).
\end{equation}
\end{lemma}

\begin{proof}
Put $v=1-p+ps$ and $u=1-p-ps$. Retaining the terms with positive even indices in the
binomial expansion in~\eqref{eq:parity-error} yields
\begin{equation}\label{eq:epsilon-integral}
 \varepsilon_{d,p}=\frac1\pi\int_0^1
 \frac{\sqrt{1-s}}{(1+s)^{3/2}}
 \bigl[v^d+u^d-2(1-p)^d\bigr]\,ds.
\end{equation}
The subtraction removes the term at time zero, where $a_0=0$.
Our restriction on $p$ ensures $|u|\le1-p$ for $0\le s\le1$.
Thus $u^d-2(1-p)^d<0$. Discard this negative term and substitute $t=v$:
\begin{align*}
 \varepsilon_{d,p}
 &<\frac1\pi\int_{1-p}^1
       \frac{\sqrt{1-t}}{(t+2p-1)^{3/2}}t^d\,dt\\
 &\le\frac{1}{\pi(2p)^{3/2}}
       \int_0^1(1-t)^{1/2}t^{d-3/2}\,dt
  =\frac{\mathrm B(d-1/2,3/2)}{\pi(2p)^{3/2}}.
\end{align*}
Here $t+2p-1\ge2pt$, since $(2p-1)(1-t)\ge0$.
Equation~\eqref{eq:J} also gives
\[
 J_d\ge\frac{\sqrt2}{\pi}\,\mathrm B(d+1,1/2).
\]
Taking the ratio of these bounds and using the beta identity
$\mathrm B(x,y)=\Gamma(x)\Gamma(y)/\Gamma(x+y)$ gives
\eqref{eq:general-budget}.

Define
\begin{equation}\label{eq:T}
 T_d=d\frac{\Gamma(d-1/2)\Gamma(d+3/2)}{\Gamma(d+1)^2}.
\end{equation}
The gamma recurrence gives
\begin{equation}\label{eq:T-recurrence}
 \frac{T_{d+1}}{T_d}=1-\frac{3}{4d(d+1)}<1,
 \qquad T_3=\frac{105\pi}{256}.
\end{equation}
Since $\pi<22/7$, we have $T_3<165/128<\sqrt2$.
For $p=1/2$ and $d\ge3$,~\eqref{eq:general-budget} therefore gives
$d\varepsilon_{d,1/2}/J_d<T_d/(2\sqrt2)<1/2$.
At $d=2$, direct evaluation gives $B_1=1$, $B_2=8/\pi-1$ and
$b_2=3/2$. The remaining case follows from
\[
 \varepsilon_{2,1/2}=\frac2\pi-\frac58
 <\frac2\pi-\frac12=\frac{J_2}{4}.\qedhere
\]
\end{proof}

\section{Proof of the BHM conjecture}\label{sec:bhmproof}

Fix one auxiliary graph $S=S_B$ from Section~\ref{sec:decimation}, and
write $g=g_B$, $d=|B|-1$, $A=A_g$, and
\begin{equation}\label{eq:bhm-rank-notation}
 q=\rk_S(A),\qquad X=d-q=K_B-1,\qquad
 \Delta=\E q-\frac d2.
\end{equation}
Lemma~\ref{lem:reflection} supplies the forest hypothesis of
Lemma~\ref{lem:pgf} with $p=1/2$. If $S$ has a cycle, we shall prove
\begin{equation}\label{eq:bhm-surplus-target}
 \Delta\ge\frac{1}{2d}.
\end{equation}
Together, Proposition~\ref{prop:comparison} and Lemma~\ref{lem:budget}
will then give $\E b_X<\E b_{Y_{d,1/2}}$. When $S$ is a tree, a direct
comparison of its independent edge increments suffices.

\subsection{Boundary values for a weighted lazy model}

When $S$ is bipartite, every vertex $w\in W$ has at most two neighbours
in $B$: three neighbours would form a triangle in $S$. Neighbourhoods of
size one contribute a constant factor of two to~\eqref{eq:marginal}.
For $e=uv\in E(S)$, let $t_e$ be the number of vertices $w\in W$ with
$N_G(w)=\{u,v\}$. Thus the marginal weight has the exact form
\begin{equation}\label{eq:edge-model}
 \PP(g)\ \propto\ \prod_{e=uv\in E(S)}w_e(g(u)-g(v)),\qquad
 w_e(k)=
 \begin{cases}
 \lambda_e=2^{t_e}\ge2,&k=0,\\
 1,&k=\pm1,\\
 0,&|k|>1.
 \end{cases}
\end{equation}
The following boundary inequality will control the probability of an
even cycle having no zero edges.

\begin{lemma}[Constant boundary values maximize the extension weight]
\label{lem:boundary}
Let $H$ be a finite simple graph, possibly disconnected, and let
$\varnothing\ne U\subseteq V(H)$ meet every connected component. For each edge $e$,
let $w_e(0)=\lambda_e\ge2$, $w_e(\pm1)=1$, and $w_e(k)=0$ for
$|k|>1$. For prescribed integer boundary values $h:U\to\Z$, set
\[
 Z_H(h)=\sum_{x:V(H)\to\Z,\ x|_U=h}
             \prod_{e=uv\in E(H)}w_e(x_u-x_v).
\]
Then
\begin{equation}\label{eq:boundary}
 Z_H(h)\le Z_H(0).
\end{equation}
\end{lemma}

\begin{proof}
Every vertex is connected to the prescribed boundary. The edge
constraints therefore bound all internal heights, so only finitely
many assignments have nonzero weight. Let $M=|V(H)|$ and
$D=\max_U h-\min_U h$. Choose an odd integer $N>2M+2D+2$ and
periodize each edge weight on $\Z/N\Z$. If
$\zeta=\exp(2\pi i/N)$, its Fourier coefficients are
\[
 W_e(j)=\lambda_e+2\cos(2\pi j/N)>0,
 \qquad
 w_{e,N}(k)=\frac1N\sum_{j\in\Z/N\Z}W_e(j)\zeta^{jk}.
\]
Orient all edges and define the divergence of an edge-frequency
assignment $\eta$ as the outgoing sum minus the incoming sum.
Writing $I=V(H)\setminus U$, the boundary partition function for this
group-valued model has the following finite expansion, obtained by
Fourier inversion and summation over the internal vertex values:
\begin{equation}\label{eq:fourier-boundary}
 Z_{H,N}(h)=N^{|I|-|E(H)|}
 \sum_{\substack{\eta\in(\Z/N\Z)^{E(H)}\\
                  \diver\eta|_I=0}}
       \left(\prod_e W_e(\eta_e)\right)
       \zeta^{\sum_{u\in U}h_u\diver\eta(u)}.
\end{equation}
The factor $N^{|I|-|E(H)|}$ accounts for one factor $N^{-1}$ per edge
and one factor $N$ per internal vertex. All coefficients are nonnegative,
and the partition function itself is nonnegative. The triangle
inequality therefore gives
$Z_{H,N}(h)\le Z_{H,N}(0)$.

To transfer this inequality to integer heights, we identify the
group-valued and integer-valued configurations. In a group-valued assignment of positive weight, each oriented
edge difference has a unique representative in $\{-1,0,1\}$.
The sum of these representatives around any simple cycle is a multiple
of $N$ and has absolute value at most the cycle length, which is less
than $N$. It is therefore zero. Decomposing closed walks into simple
cycles and backtracks shows that these increments can be integrated
to integer heights on each component.

Fix one boundary vertex in a component at its prescribed integer height.
The lift is then unique. For any other boundary vertex $v$, its lifted
height $\widetilde h_v$ differs from $h_v$ by a multiple of $N$; along a
simple path from the chosen boundary vertex,
\[
 |\widetilde h_v-h_v|\le M-1+D<N.
\]
Thus the difference is zero. This establishes a weight-preserving
bijection, whose inverse is reduction modulo $N$. The same construction
applies to the zero boundary. Consequently, $Z_{H,N}(h)=Z_H(h)$ and
$Z_{H,N}(0)=Z_H(0)$, proving~\eqref{eq:boundary}.
\end{proof}

\subsection{A cycle supplies a rank surplus}

The next lemma allows us to extend a rank estimate on one cycle to
the whole graph, without any independence assumption.

\begin{lemma}[Extending a local forest]\label{lem:extend}
Let $S$ be connected and let $C$ be a connected subgraph with vertex set
$U$. Contract all of $U$ to one vertex, choose a spanning tree of the
resulting quotient, and choose one original representative for each tree
edge. Denote this fixed set of representatives by $F$. Then
$|F|=|V(S)|-|U|$, and for every $A\subseteq E(S)$,
\begin{equation}\label{eq:extend}
 \rk_S(A)\ge\rk_C(A\cap E(C))+|A\cap F|.
\end{equation}
\end{lemma}

\begin{proof}
Take a maximal forest $J$ of $A\cap E(C)$ and extend it to a spanning
tree $T_C$ of $C$. The graph $T_C\cup F$ is connected with
$|V(S)|-1$ edges, so it is a spanning tree of $S$. Its subgraph
$J\cup(A\cap F)$ is a forest contained in $A$, proving
\eqref{eq:extend}. Crucially, $F$ is fixed before $A$ is sampled.
\end{proof}

\begin{proposition}[Rank surplus in an auxiliary graph]\label{prop:bhm-surplus}
For the single-class marginal~\eqref{eq:marginal}, if $S$ contains a
cycle, then $d\ge2$ and $\Delta\ge1/(2d)$.
\end{proposition}

\begin{proof}
Every edge of $S$ is zero with probability at least $1/2$ by
Lemma~\ref{lem:reflection}. Suppose first that $S$ has an odd cycle $C$
of length $\ell\ge3$. Let $Z_C=|A\cap E(C)|$ and
$q_C=\rk_C(A\cap E(C))$. The number of nonzero increments around
a closed integer-height cycle is even. Hence $Z_C$ is positive and odd,
and
\[
 q_C=Z_C-\one_{\{Z_C=\ell\}}.
\]
For every configuration,
\begin{equation}\label{eq:odd-certificate}
 q_C\ge\frac{(\ell-2)Z_C+1}{\ell-1}.
\end{equation}
Indeed, equality holds for $Z_C=\ell$, and otherwise the difference,
after multiplication by $\ell-1$, is $Z_C-1\ge0$.
Since $\E Z_C\ge\ell/2$, this yields
\begin{equation}\label{eq:odd-rank}
 \E q_C\ge\frac{\ell-1}{2}+\frac1{2(\ell-1)}.
\end{equation}

It remains to consider the case where $S$ is bipartite and has a cycle.
The edge model
\eqref{eq:edge-model} applies. Choose an even cycle $C$ of length
$\ell\ge4$, remove its edges, and take $V(C)$ as the boundary of the
remaining graph $H$. Each component of $H$ meets this boundary:
a shortest path to $C$ uses no cycle edge before reaching $C$.
Since the weights depend only on height differences, we may reroot
at a vertex of $C$. Let
\[
 P_0=\PP(Z_C=\ell),\qquad P_*=\PP(Z_C=0).
\]
The unnormalized weight of the all-zero event is
$\bigl(\prod_{e\in E(C)}\lambda_e\bigr) Z_H(0)$.
On the event of no zero edges, the cycle has
$\binom{\ell}{\ell/2}$ possible rooted height assignments, each with
cycle-edge weight one. Lemma~\ref{lem:boundary} bounds each extension
weight by $Z_H(0)$. Comparing the two unnormalized event weights gives
\begin{equation}\label{eq:cycle-events}
 P_*\le\frac{\binom{\ell}{\ell/2}}{\prod_{e\in E(C)}\lambda_e}P_0
      \le\frac12 P_0.
\end{equation}
The last bound uses $\lambda_e\ge2$ and the fact that
$4^{-r}\binom{2r}{r}$ decreases from $1/2$ at $r=1$.

Here $Z_C$ is even. If $0<Z_C<\ell$, then $Z_C\ge2$ and $q_C=Z_C$.
Together with the cases $Z_C=0,\ell$, this gives the pointwise inequality
\begin{equation}\label{eq:even-certificate}
 q_C\ge
 \frac{(\ell-4)Z_C+2+\one_{\{Z_C=\ell\}}
                        -2\one_{\{Z_C=0\}}}{\ell-3}.
\end{equation}
At the two endpoints this is equality; otherwise the difference after
multiplication by $\ell-3$ is $Z_C-2$.
Taking expectations, applying~\eqref{eq:cycle-events}, and using
$\E Z_C\ge\ell/2$, we obtain
\begin{equation}\label{eq:even-rank}
 \E q_C\ge\frac{\ell-1}{2}+\frac1{2(\ell-3)}.
\end{equation}
The coefficient $\ell-4$ is nonnegative, including at $\ell=4$.

To pass from either cycle estimate to the full graph, choose the fixed
set $F$ from Lemma~\ref{lem:extend}.
It has $d-(\ell-1)$ edges, each zero with probability at least $1/2$.
Thus
\[
 \E q\ge\E q_C+\frac{d-(\ell-1)}2.
\]
Equations~\eqref{eq:odd-rank} and~\eqref{eq:even-rank} give respectively
$\Delta\ge1/[2(\ell-1)]$ and $\Delta\ge1/[2(\ell-3)]$.
Both bounds are at least $1/(2d)$. Only the marginal zero probabilities
of the additional edges were used; independence from the cycle is
unnecessary.
\end{proof}

\subsection{The single-class comparison and the induction}

\begin{proposition}[Single-class comparison]\label{prop:single-class}
For either bipartition class, with $d$ one less than its size and $K$
the number of components of its full zero-edge subgraph,
\begin{equation}\label{eq:single-class}
 \E b_{K-1}\le\E b_{Y_{d,1/2}}.
\end{equation}
The inequality is strict if the corresponding auxiliary graph contains
a cycle.
\end{proposition}

\begin{proof}
For a cyclic auxiliary graph, Proposition~\ref{prop:comparison} with
$p=1/2$, together with Proposition~\ref{prop:bhm-surplus} and
\eqref{eq:half-budget}, yields
\[
 \E b_{K-1}-\E b_{Y_{d,1/2}}
 \le\varepsilon_{d,1/2}-\frac{J_d}{2d}-\E a_{K-1}<0.
\]
If the auxiliary graph is a tree, its marginal has the edge-factorized
form~\eqref{eq:edge-model}. On a rooted tree, height functions are in
bijection with edge-increment assignments. The factorization makes
these increments independent,
and edge $e$ is nonzero with probability
$2/(\lambda_e+2)\le1/2$. On a tree the zero-edge rank equals the number
of zero edges, so $K-1$ is the sum of these independent nonzero indicators.
Couple each indicator below an independent Bernoulli$(1/2)$ variable.
Since $b_m$ is increasing,~\eqref{eq:single-class} follows. A one-vertex
auxiliary graph has $d=0$, with both sides zero. The tree argument
therefore also covers the small cases excluded from the cycle argument.
\end{proof}

\begin{proof}[Proof of Theorem~\ref{thm:bhm}]
We proceed by induction on $n=|V(G)|$. If $n=1$, both ranges are zero.
Suppose
$n\ge2$ and that the theorem holds for all smaller connected bipartite
graphs. Let $|B|=a$ and $|W|=b$, so $a+b=n$. Put
\[
 t_j=\E b_{Y_{j,1/2}}\qquad(j\ge0).
\]
Each quotient in~\eqref{eq:decimation} has fewer than $n$ vertices, so
the induction hypothesis applies to every realized quotient. Applying
Proposition~\ref{prop:single-class} then gives
\begin{equation}\label{eq:partition-bound}
 \widehat h(G)
 \le1+\E b_{K_B-1}+\E b_{K_W-1}
 \le1+t_{a-1}+t_{b-1}.
\end{equation}

To compare the two bipartition sizes, we use the concavity of $(t_j)$.
Couple $Y_{j+1,1/2}=Y_{j,1/2}+\xi$, where the additional fair Bernoulli
variable $\xi$ is independent. This coupling gives
\[
 t_{j+1}-t_j
   =\frac12\E\bigl[b_{Y_{j,1/2}+1}-b_{Y_{j,1/2}}\bigr].
\]
The increments of $b$ are nonincreasing, and the coupled binomial
variables are nondecreasing in $j$. Hence these increments of $t$ are
nonincreasing. Balancing two nonnegative indices with sum $n-2$ gives
\begin{equation}\label{eq:balance}
 t_{a-1}+t_{b-1}
 \le t_{\lfloor(n-2)/2\rfloor}+t_{\lceil(n-2)/2\rceil}.
\end{equation}

To identify the balanced bound, apply the exact
decomposition~\eqref{eq:decimation} to $P_n$.
Its two auxiliary graphs are paths, including the possible one-vertex
path, and every auxiliary edge has $\lambda_e=2$. Their nonzero-edge
counts are binomial with parameter $1/2$, and each zero-edge quotient
is again a path. Thus
\begin{equation}\label{eq:path-decimation}
 b_{n-1}=1+t_{\lfloor(n-2)/2\rfloor}
           +t_{\lceil(n-2)/2\rceil}.
\end{equation}
Combining~\eqref{eq:partition-bound}--\eqref{eq:path-decimation} proves
$\widehat h(G)\le b_{n-1}=\widehat h(P_n)$ and completes the induction.
\end{proof}

The sharper estimate~\eqref{eq:partition-bound} retains the two
bipartition sizes. This completes the proof of BHM without invoking
the LNR inequality.

For the running example, the path formula gives
$\widehat h(P_5)=b_4=19/8$, whereas $\widehat h(G_\star)=9/5$.
The triangle $S_W$ is the cyclic auxiliary graph to which the strict
single-class comparison applies.

\section{BHM implies LNR}\label{sec:lnr}

We now consider a uniform lazy height function $f\in\LL(G)$ on an
arbitrary connected simple graph with $n=d+1$ vertices. Let
\begin{equation}\label{eq:lazy-notation}
 A=A_f,\qquad K=k_G(A),\qquad q=\rk_G(A),\qquad
 X=K-1=d-q,\qquad \Delta_G=\E q-\frac d3.
\end{equation}
Deleting the zero increments from the lazy walk on a path leaves a
simple random walk. Conditional on the number of remaining increments,
their signs are independent and fair. Consequently,
\begin{equation}\label{eq:lazy-path}
 h(P_n)=\E b_{Y_{d,2/3}}.
\end{equation}
We first compare the number of zero-edge components with the path
model. This comparison will be proved independently of
Theorem~\ref{thm:bhm}.

\begin{proposition}[Zero-edge contraction comparison]\label{prop:lazy-contraction}
For the uniform lazy model on every connected simple graph $G$,
\begin{equation}\label{eq:lazy-contraction}
 \E b_{K-1}\le h(P_n).
\end{equation}
Equality holds if and only if $G$ is a tree.
\end{proposition}

We prove the proposition in three steps: a forest probability bound,
a rank surplus for cyclic graphs, and the corresponding parity estimate.

\subsection{A switching argument for zero edges}

\begin{lemma}\label{lem:lazy-forest}
In the uniform lazy model,
\begin{equation}\label{eq:lazy-edge}
 \PP(f(u)=f(v))\ge\frac13\qquad(uv\in E(G)).
\end{equation}
For every forest $F\subseteq E(G)$,
\begin{equation}\label{eq:lazy-forest}
 \PP(F\subseteq A)\ge3^{-|F|}.
\end{equation}
\end{lemma}

\begin{proof}
For distinct vertices $u,v$, let
\[
 c_j=|\{f\in\LL(G,u):f(v)=j\}|\qquad(j\in\Z).
\]
We prove $c_j^2\ge c_{j-1}c_{j+1}$ by an injection. Take a pair of
integer 1-Lipschitz functions $f,g$ with
\[
 (f(u),f(v))=(0,j+1),\qquad (g(u),g(v))=(1,j).
\]
Translating the second function down by one shows that there are
$c_{j+1}c_{j-1}$ such pairs. At every vertex $x$, set
$\ell_x=\min(f(x),g(x))$ and $h_x=\max(f(x),g(x))$.
Pointwise minima and maxima preserve the Lipschitz constraint, so both
$\ell$ and $h$ are integer 1-Lipschitz functions.
On the vertices where $\ell_x<h_x$, record whether $f$ takes the upper
or lower value. Join two such vertices $x,y$ by a \emph{locking edge}
when $xy\in E(G)$ and at least one of
\[
 |\ell_x-h_y|\le1,\qquad |h_x-\ell_y|\le1
\]
fails. Equal choices of upper or lower values are always legal, since
$\ell$ and $h$ are Lipschitz. Opposite choices are legal exactly when
the edge is not locking. Edges incident to a vertex with equal upper
and lower values impose no additional choice constraint: exchanging
the two functions merely permutes the same pair of edge constraints.

It follows that legal pairs with these fixed unordered vertex values
are precisely the choices of a constant sign on each locking component.
At the vertex $u$, the function $f$ takes the lower value; at $v$, it
takes the upper value. These vertices therefore lie in distinct locking
components. Exchange $f$ and $g$ on the component containing $v$.
The new pair $f',g'$ is legal and satisfies
\[
 (f'(u),f'(v))=(0,j),\qquad
 (g'(u),g'(v))=(1,j+1).
\]
The unordered vertex values, and hence the locking graph, remain
unchanged, so exchanging the same component again recovers the input.
The map is therefore injective. After translating the second function
down by one, we see that the target has $c_j^2$ elements. This proves
the claimed log-concavity.

Negating all heights gives $c_{-1}=c_1$. Combining this symmetry with
log-concavity gives $c_0\ge c_1$.
If $u,v$ are adjacent, their difference can only be $-1,0,1$; hence
\[
 \PP(f(u)=f(v))=\frac{c_0}{c_0+2c_1}\ge\frac13.
\]
To obtain the forest bound, require the forest edges to be zero one
at a time.
Conditioning on the previous equalities gives the uniform lazy model
on the graph obtained by contracting those edges and deleting loops
and repeated edges. A remaining forest edge still has distinct
endpoints. The conditional zero probability is at least $1/3$ at
each step, which proves~\eqref{eq:lazy-forest}.
\end{proof}

\subsection{Rank surplus in the uniform lazy model}

\begin{lemma}\label{lem:lazy-surplus}
If $G$ contains a cycle, then $d\ge2$ and
\begin{equation}\label{eq:lazy-surplus}
 \Delta_G\ge\frac{2}{7d}.
\end{equation}
\end{lemma}

\begin{proof}
First suppose that $G$ is 2-connected and has at least three vertices.
We first show that the full zero set is nonempty and proper with
uniformly positive probability:
\begin{equation}\label{eq:partial-zero}
 \PP(\varnothing\ne A\ne E(G))\ge\frac47.
\end{equation}
To prove this, note that all rooted assignments with values in
$\{0,1\}$ or in $\{0,-1\}$ are lazy. These two families have
$2^{d+1}-1$ distinct members. The event $A=E(G)$ consists of the single
constant-zero assignment. If $G$ is not bipartite, the event
$A=\varnothing$ is impossible, since it would give a standard height
function. Since $d\ge2$, the probability of a nonempty proper zero set
is therefore at least $1-(2^{d+1}-1)^{-1}\ge6/7$.

If $G$ is bipartite, let $M$ be its number of rooted standard height
functions and $I$ its number of lazy height functions with partial zero
set. Reading signs on a spanning tree gives $M\le2^d$. Among the
$2^{d+1}-1$ two-valued assignments just counted, exactly one has all
edges zero and exactly two have no zero edge. The latter assertion
uses connectedness and the uniqueness of the bipartition up to exchange.
Therefore $I\ge2^{d+1}-4$, and
\[
 \PP(\varnothing\ne A\ne E(G))
   =\frac{I}{I+M+1}
   \ge\frac{2^{d+1}-4}{3\cdot2^d-3}\ge\frac47.
\]
Here $d\ge3$, since a bipartite 2-connected simple graph has at least
four vertices. This proves~\eqref{eq:partial-zero}.

Fix a spanning tree $S_0$ of $G$ and orient each of its $d$ edges in
both directions. For each directed edge $v\to w$, choose a spanning
tree of $G-v$, which exists by 2-connectedness, and add $vw$. Denote
the resulting spanning tree by
$T_{v\to w}$. It makes $v$ a leaf adjacent only to $w$.
These $2d$ trees are fixed and need not be distinct.

For every full zero set $A$, each $T_{v\to w}\cap A$ is a forest in
$A$, so it has at most $\rk_G(A)$ edges. If
$\varnothing\ne A\ne E(G)$, there is a zero component $C$ containing
at least two vertices, and $C\ne V(G)$. The latter property uses that
$A$ is the \emph{full} zero set: a connected zero subgraph would force
the entire function to be constant. Some edge of $S_0$ leaves $C$;
orient it $v\to w$ with $v\in C$ and $w\notin C$. In
$T_{v\to w}\cap A$, the vertex $v$ is isolated. This forest therefore
fails to connect the vertices of $C$ and has at most $\rk_G(A)-1$ edges.
We have proved the pointwise inequality
\begin{equation}\label{eq:detector}
 \rk_G(A)-\frac1{2d}\sum_{v\to w\in\vec E(S_0)}
                          |A\cap E(T_{v\to w})|
 \ge\frac{\one_{\{\varnothing\ne A\ne E(G)\}}}{2d}.
\end{equation}
By~\eqref{eq:lazy-edge}, each tree in the sum has at least $d/3$
zero edges in expectation. Taking expectations in~\eqref{eq:detector}
and applying~\eqref{eq:partial-zero} proves
$\Delta_G\ge2/(7d)$ for 2-connected $G$.

For a general connected graph, decompose it into blocks, including each
bridge as a block isomorphic to $K_2$. Write the blocks as $G_i$ and
$d_i=|V(G_i)|-1$, so $d=\sum_i d_i$. Root each block at the articulation
vertex toward the global root, with an arbitrary local root in the
initial block. Subtracting the local root height gives a bijection
between global lazy functions and tuples of rooted lazy functions on
the blocks. Conversely, translate the block functions successively
along the block tree to recover the global function. The local
functions are therefore independent and uniform.

The zero-edge rank is additive under this decomposition. Gluing two
graphs at a single vertex reduces both the sum of their vertex counts
and the sum of their zero-component counts by one. Their ranks
therefore add. Applying this observation successively to the blocks gives
\begin{equation}\label{eq:block-ranks}
 q_G=\sum_i q_{G_i},\qquad
 \Delta_G=\sum_i\Delta_{G_i}.
\end{equation}
A bridge block has $\Delta_{G_i}=0$. Each other block is 2-connected
with at least three vertices and contributes at least $2/(7d_i)$.
If $G$ has a cycle, there is at least one such block, with $d_i\le d$.
Equation~\eqref{eq:lazy-surplus} follows.
\end{proof}

\subsection{The lazy comparison and the implication}

\begin{lemma}\label{lem:lazy-budget}
For every integer $d\ge2$,
\begin{equation}\label{eq:lazy-budget}
 \varepsilon_{d,2/3}<\frac{2J_d}{7d}.
\end{equation}
\end{lemma}

\begin{proof}
Use the quantities $T_d$ from~\eqref{eq:T}. For $p=2/3$,
Lemma~\ref{lem:budget} gives
\[
 \frac{d\varepsilon_{d,2/3}}{J_d}<\frac{3\sqrt6}{32}T_d.
\]
The recurrence~\eqref{eq:T-recurrence} shows that the sequence decreases, with
$T_4=1575\pi/4096$. For $d\ge4$, using $\sqrt6<5/2$ and $\pi<22/7$,
\[
 \frac{3\sqrt6}{32}T_d
 \le\frac{3\sqrt6}{32}T_4
 <\frac{15}{64}\frac{2475}{2048}
 =\frac{37125}{131072}<\frac27.
\]
The last inequality reduces to the integer comparison $259875<262144$.
For $d=2,3$, use $B_1=1$, $B_2=8/\pi-1$, $B_3=b_3=2$ and
$b_2=3/2$ to obtain
\begin{align*}
 \varepsilon_{2,2/3}-\frac{J_2}{7}
     &=\frac{152}{63\pi}-\frac{52}{63}<0,\\
 \varepsilon_{3,2/3}-\frac{2J_3}{21}
     &=\frac{272}{63\pi}-\frac{88}{63}<0.
\end{align*}
The two inequalities follow respectively from $\pi>38/13$ and
$\pi>34/11$, both consequences of $\pi>31/10$. This covers the remaining
values of $d$.
\end{proof}

\begin{proof}[Proof of Proposition~\ref{prop:lazy-contraction}]
Suppose first that $G$ contains a cycle. Lemma~\ref{lem:lazy-forest}
verifies the hypotheses of Proposition~\ref{prop:comparison} with
$p=2/3$. Applying the rank and parity estimates from
Lemmas~\ref{lem:lazy-surplus} and~\ref{lem:lazy-budget}, and identifying
the path expectation through~\eqref{eq:lazy-path}, we obtain
\[
 \E b_{K-1}-h(P_n)
 \le\varepsilon_{d,2/3}-\Delta_G J_d-\E a_{K-1}
 \le\varepsilon_{d,2/3}-\frac{2J_d}{7d}-\E a_{K-1}<0.
\]
If $G$ is a tree, its rooted lazy functions are in bijection with all
assignments of increments in $\{-1,0,1\}$ to its oriented edges.
Thus the increments are independent and uniform. Its zero-edge rank
equals the number of zero edges, giving $K-1\sim\Bin(d,2/3)$ and
equality by~\eqref{eq:lazy-path}. This includes the one-vertex tree.
\end{proof}

\begin{proposition}[The implication BHM $\Longrightarrow$ LNR]
\label{prop:implication}
If the BHM expectation inequality holds for every finite connected
bipartite simple graph, then the LNR expectation inequality holds for
every finite connected simple graph. The latter is strict on graphs
containing a cycle.
\end{proposition}

\begin{proof}
Since $f\in\LL(G)$ is uniform, its sampling weight is constant.
Conditioning on its full zero set and applying
Lemma~\ref{lem:fullzero} gives
\[
 h(G)=\E\widehat h(G/A_f).
\]
For every realized zero set, $G/A_f$ is a connected bipartite graph on
$K$ vertices. We may therefore apply the assumed BHM inequality to
each quotient. Proposition~\ref{prop:lazy-contraction} then completes
the comparison:
\begin{equation}\label{eq:implication}
 h(G)=\E\widehat h(G/A_f)
      \le\E b_{K-1}\le h(P_n).
\end{equation}
The last inequality is strict when $G$ has a cycle.
\end{proof}

\begin{proof}[Proof of Corollary~\ref{cor:lnr}]
Theorem~\ref{thm:bhm}, already proved in Section~\ref{sec:bhmproof},
supplies the hypothesis of Proposition~\ref{prop:implication}.
\end{proof}

\paragraph{Running example: the implication on $K_{2,3}$.}
Finally, sample $f$ uniformly from $\LL(G_\star,u_0)$.
Here the full zero set is taken in the original five-vertex graph.
There are $45$ rooted lazy functions: fixing $f(u_1)$ to be $\pm2$
gives one function for each sign, fixing it to be $\pm1$ gives
eight for each sign, and fixing it to be zero gives $3^3=27$.
Their zero-edge quotients are listed in Table~\ref{tab:running-lazy}.
For example, if $f(u_1)=0$ and exactly one white vertex has height
zero, its two incident zero edges merge the two black vertices with
that white vertex. The quotient is $P_3$, and there are
$3\cdot2^2=12$ such functions. If exactly two white vertices have
height zero, there are $3\cdot2=6$ functions with quotient $P_2$;
the $16$ functions with $f(u_1)=\pm1$ also have this quotient.
The constant function gives $P_1$, and the ten standard functions
have no zero edges and retain $G_\star$ itself.

\begin{table}[htbp]
\centering
\caption{Zero-edge contraction for the uniform lazy model on
$G_\star=K_{2,3}$. The count $N_K$ is the number of rooted height
functions with $K$ zero-edge components. For each row, all quotients
are isomorphic to the displayed graph.}
\label{tab:running-lazy}
\smallskip
\begin{tabular}{ccccc}
\toprule
$K$ & $q=5-K$ & $G_\star/A_f$ & $N_K$ &
 $\E[R(f)\mid K]$ \\
\midrule
$1$ & $4$ & $P_1$       & $1$  & $0$ \\
$2$ & $3$ & $P_2$       & $22$ & $1$ \\
$3$ & $2$ & $P_3$       & $12$ & $3/2$ \\
$5$ & $0$ & $K_{2,3}$   & $10$ & $9/5$ \\
\bottomrule
\end{tabular}
\end{table}

By Lemma~\ref{lem:fullzero}, the last column consists of the standard
expected ranges of the quotients. Averaging gives the first term
below; replacing each quotient by the path of the same order gives
the middle term. Thus~\eqref{eq:implication} becomes
\[
 h(G_\star)=\frac{58}{45}
   <\E b_{K-1}=\frac{17}{12}
   <\frac{146}{81}=h(P_5).
\]
Indeed, the middle numerator before division by $45$ is
$22b_1+12b_2+10b_4$, and~\eqref{eq:lazy-path} gives the final value.
The first inequality is strict because the no-zero event retains
$K_{2,3}$, whose standard expected range is strictly below that of
$P_5$. The second is the contraction comparison on this cyclic graph.

\begin{remark}
The equality assertion in Proposition~\ref{prop:lazy-contraction}
concerns $\E b_{K-1}$, rather than $h(G)$. For a tree, the quotient
$G/A_f$ need not be a path, so the first inequality in
\eqref{eq:implication} can still be strict. In particular, the
contraction comparison does not assert equality in LNR for every tree.
\end{remark}

\subsection*{Use of artificial intelligence}
The proof was obtained through interaction with OpenAI GPT-6 Astra
and verified by the author. OpenAI Codex was used to audit the arguments
and develop the accompanying Lean~4 formalization. The author takes
full responsibility for the mathematical content and the final manuscript.

\bibliographystyle{amsplain}
\bibliography{references}
\end{document}